\documentclass[reqno,12pt]{amsart}%
\usepackage{amsfonts}
\usepackage{amsmath}
\usepackage{amssymb}
\usepackage{graphicx}%
\providecommand{\U}[1]{\protect\rule{.1in}{.1in}}
\newtheorem{theorem}{Theorem}
\theoremstyle{plain}

\newtheorem{definition}{Definition}
\newtheorem{example}{Example}

\newtheorem{proposition}{Proposition}

\numberwithin{equation}{section}
\begin{document}
\title[Screen Almost Semi-Invariant Lightlike Submanifolds of indefinite Kaehler Manifolds]{Screen Almost Semi-Invariant Lightlike Submanifolds of indefinite Kaehler Manifolds}
\author{Sema\ KAZAN}
\curraddr{Department of Mathematics, Faculty of Arts and Sciences, Inonu University,
Malatya, Turkey}
\email{sema.bulut@inonu.edu.tr}
\author{Cumali\ YILDIRIM}
\curraddr{Department of Mathematics, Faculty of Arts and Sciences, Inonu University,
Malatya, Turkey}
\email{cumali.yildirim@inonu.edu.tr}
\thanks{}
\date{2022}

\begin{abstract}
In the present paper, we introduce screen almost semi-invariant (SASI)
lightlike submanifolds of indefinite Keahler manifolds. We obtain the
neccesary and sufficient condition for the induced connection to be a metric
connection on SASI-lightlike submanifolds and construct an example for this
manifold. Also we find some conditions for integrability of distributions and
investigate certain characterizations.

\end{abstract}
\maketitle

\bigskip

\section{\textbf{INTRODUCTION}}

\bigskip

\bigskip

One of the most important issues of differential geometry is the Riemannian
geometry of submanifolds \cite{Chen}. Obviously, semi-Riemannian submanifolds
have similar properties with Riemannian submanifolds, but the lightlike
submanifolds \cite{Bejancu} are different since (contrary to the nondegenerate
cases) their normal vector bundle intersects with the tangent bundle. So,
studying them becomes more difficult than studying non-degenerate
submanifolds. When we think of hypersurfaces as submanifolds, we say that
lightlike hypersurfaces of semi-Riemannian manifolds are important due to
their physical applications in mathematical physics. Moreover, in physics,
lightlike hypersurfaces are interesting in general relativity since they
produce models of different types of horizons. Lightlike hypersurfaces are
also studied in the theory of electromagnetism.

The geometry of lightlike submanifolds of a semi-Riemannian manifold has been
presented in \cite{Bejancu} (see also \cite{Bejancu2}) by Duggal and Bejancu.
In \cite{Bayrm}, Duggal and \c{S}ahin have introduced differential geometry of
lightlike submanifolds and they have studied geometry of classes of lightlike
submanifolds (see \cite{Bayrm2}, \cite{Bayrm3}, \cite{Bayrm4}, \cite{Bayrm5}).
Also, the geometry of lightlike submanifolds of indefinite Kaehler manifolds
has been presented in a book by Duggal and Bejancu \cite{Bayrm4}. The notion
of semi-invariant lightlike submanifolds has been studied some authors. For
instance, authors have introduced screen semi invariant lightlike submanifolds
of semi-Riemannian product manifolds in \cite{Erol} and they have given the
following definition:

Let $(\overline{M},\overline{g})$ be a semi-Riemannian product manifold and
$M$ be a lightlike submanifold of $\overline{M}$. We say that $M$ is screen
semi-invariant (SSI)-lightlike submanifold of $\overline{M}$ if the following
statements are satisfield

\ 1) There exists a non-null distribution $D\subseteq S(TM)$ such that%
\[
S(TM)=D\bot D^{\bot},\text{ }FD=D,\text{ }FD^{\bot}\subset S(TM^{\bot}),\text{
\ }D\cap D^{\bot}=\{0\},
\]
where $D^{\bot}$ is orthogonal complementary to $D$ in $S(TM)$.

2) $Rad(TM)$ is invariant with respect to $F$, that is $FRad(TM)=Rad(TM)$.

Then we have%
\begin{align*}
F\ell tr(TM)  &  =\ell tr(TM),\\
TM  &  =D^{^{\prime}}\bot D^{\bot}\text{, }D^{^{\prime}}=D\bot Rad(TM).
\end{align*}

Hence, it follows that $D^{^{\prime}}$ is also invariant with respect to $F$.
We denote the orthogonal complement to $FD^{\bot}$ in $S(TM^{\bot})$ by
$D_{0}$. Then, we have%
\[
tr(TM)=ltr(TM)\bot FD^{\bot}\bot D_{0},
\]
where $F^{2}=I$\ and\ $\overline{g}(FX,Y)=\overline{g}(X,FY),$ for any
$X,Y\in(T\overline{M}).$

Also, Bahad\i r has studied screen semi-invariant half-lightlike submanifolds
of a semi-Riemannian product manifold \cite{Oguz}. In \cite{Nergiz},
semi-invariant lightlike submanifolds of golden semi-Riemannian manifolds have
been introduced by Poyraz and Do\u{g}an. And, in \cite{Atckn}, authors have
introduced semi-invariant lightlike submanifolds of a semi-Riemannian product
manifold. Finally, Gupta and friends have given geometry of semi-invariant
lightlike product manifolds in \cite{Garima}.

On the other hand, in 1984, Bejancu and Papaghiuc \cite{Bejancu3} have given
the definition of almost semi-invariant submanifolds of a Sasakian manifold as follows:

Let $\widetilde{M}$ be a $(2n+l)-$dimensional almost contact metric manifold
with $(F,\xi,\eta,g)$ as the almost contact metric structure, where $F$ is
tensor field of type $(1,1)$, $\xi$ is a vector field, $\eta$ is a $1$-from
and $g$ is a Riemannian metric on $\widetilde{M}$. These tensor fields satisfy%
\begin{equation}
F^{2}=-I+\eta\otimes\xi,\text{ \ \ }F(\xi)=0,\text{ \ \ \ \ }\eta
(\xi)=1,\text{ \ \ \ }\eta\circ F=0 \label{1}%
\end{equation}
and
\begin{equation}
g(FX,FY)=g(X,Y)-\eta(X)\eta(Y) \label{1'}%
\end{equation}
\bigskip for all vector fields $X,Y$ tangent to $\widetilde{M}$, where $I$
designes the identity morphism on the tangent bundle $T\widetilde{M}.$ It is
well known that $\widetilde{M}$ is a Sasakian manifold if and only\bigskip\
\begin{equation}
(\widetilde{\nabla}_{X}F)Y=g(X,Y)\xi-\eta(Y)X \label{1''}%
\end{equation}
for all $X,Y$ tangent to $\widetilde{M}$, where $\widetilde{\nabla}$ is the
Riemannian connection with respect to $g.$ From (\ref{1''}), we have
\begin{equation}
\widetilde{\nabla}_{X}\xi=-FX. \label{1'''}%
\end{equation}

Now, let $M$ be an $m$-dimensional manifold isometrically immersed in a
Sasakian manifold $\widetilde{M}$. Denote by $TM$ and $TM^{\bot}$ the tangent
bundle of $M$ and the normal bundle to $M,$ respectively. Suppose the
structure vector field $\xi$\ of $\widetilde{M}$ be tangent to the submanifold
$M$ and denote by $\{\xi\}$ the $1$-dimensional distribution spanned by $\xi$
on $M$ and by $\{\xi\}^{\bot}$ the complementary orthogonal distribution to
$\{\xi\}$ in $TM.$

For any vector bundle $H$ on $M,$ we denote by $\Gamma(H)$ the module of all
differentiable sections of $H$. For any $X$ $\in\Gamma(TM),$ we have
$g(FX,\xi)=0$. Then we put%
\begin{equation}
FX=bX+cX, \label{1''''}%
\end{equation}
where \ $bX\in\Gamma(\{\xi\}^{\bot})$ and $cX\in\Gamma(TM^{\bot}).$ Thus $b$
is an endomorphism of the tangent bundle $TM$ and $c$ is a normal bundle
valued $1$-form on $M$. Next, for each $x\in M$ we define the following
subspaces:%
\begin{equation}
D_{x}=\left\{  X_{x}\in\{\xi\}_{x}^{\bot}:c(X_{x})=0\right\}  \label{1'1}%
\end{equation}
and%
\begin{equation}
D_{x}^{\bot}=\left\{  X_{x}\in\{\xi\}_{x}^{\bot}:b(X_{x})=0\right\}  .
\label{1'11}%
\end{equation}

We note that $D_{x}$ and $D_{x}^{\bot}$ are two orthogonal subspaces of space
$T_{x}M$. In fact, using (\ref{1}), (\ref{1'}) and (\ref{1''''}), for any
$X_{x}\in$ $D_{x}$ and $Y_{x}\in D_{x}^{\bot},$ we have
\[
g(X_{x},Y_{x})=g(FX_{x},FY_{x})=g(bX_{x},cY_{x})=0.
\]

Thus, the submanifold $M$ of the Sasakian manifod $\widetilde{M}$ is said to
be an \textit{almost semi-invariant submanifold} if $dim(D_{x})$ and
$dim(D_{x}^{\bot})$ are constant along $M$ and $D:x\rightarrow D_{x}\subset
T_{x}M,$ $D^{\bot}:x\rightarrow D_{x}^{\bot}\subset T_{x}M$ define
differentiable distributions on $M$.

Now we denote by $\widetilde{D}$ the complementary orthogonal to $D\oplus
D^{\bot}\oplus\xi$ in $TM$. Thus, for the tangent bundle to the almost
semi-invariant submanifold $M$, we have the orthogonal decomposition:%
\begin{equation}
TM=D\oplus D^{\bot}\oplus\widetilde{D}\oplus\xi, \label{1'1'}%
\end{equation}
where

\textit{i)} $D$ is an invariant distribution on $M$, that is $FD=D,$

\textit{ii)} $D^{\bot}$ is anti-invariant distribution on $M$, that is
$FD^{\bot}\subset TM^{\bot},$

\textit{iii)} $\widetilde{D}$ is neither an invariant nor an anti-invariant
distribution on $M,$ that is $bX_{x}\neq0$ and $cX_{x}\neq0,$ for any $x\in M$
and $X_{x}\in\widetilde{D}_{x}.$

Also, in \cite{Papaghiuc}, Papaghiuc has introduced some results on almost
semi-invariant submanifolds in Sasakian manifolds. For similar studies, see
\cite{Bejan}, \cite{Cabras}, \cite{Hasan} and \cite{Tripathi}.

In this paper, after giving some basic notions about lightlike submanifolds in
the second Section, in Section 3 we give the definition of screen almost
semi-invariant (SASI) lightlike submanifolds of indefinite Keahler manifolds
and construct an example for this submanifold. After as a characteristic
theorem, we obtain the neccesary and sufficient condition for the induced
connection to be a metric connection on SASI-lightlike submanifolds. Also, we
give some certain characterizations for SASI-lightlike submanifolds. For
instance, we investigate the totally geodesic of the SASI-lightlike submanifolds.

\section{\textbf{LIGHTLIKE SUBMANIFOLDS}}

Let $(\overline{N}^{p+q},\overline{\rho})$ be a semi-Riemannian manifold and
let $\widetilde{N}^{p}$ be an immersed submanifold in $\overline{N}.$ If
$\widetilde{N}$ is a lightlike manifold with respect to the metric $\rho$
induced from $\overline{\rho}$ and the radical distribution $Rad(T\widetilde
{N})$ is of $ranks$, where $1\leq s\leq p,$ then $\widetilde{N}$ is called a
lightlike submanifold \cite{Bejancu}. According to this definition, we recall
basic notions about lightlike submanifolds as following. A semi-Riemannian
complementary distribution of $Rad(T\widetilde{N})$ in $T\widetilde{N}$ is
called screen distribution and it is denoted by $S(T\widetilde{N}),$ that is
$T\widetilde{N}=Rad(T\widetilde{N})\bot S(T\widetilde{N}).$ Now, we consider a
screen transversal vector bundle $S(T\widetilde{N}^{\bot}).$ This vector
bundle is a semi-Riemannian complementary vector bundle of $Rad(T\widetilde
{N})$ in%
\begin{equation}
T\widetilde{N}^{\bot}=\cup_{x\in\widetilde{N}}\{u\in T_{x}\overline{N}\text{
}|\text{\ }\overline{\rho}(u,v)=0,\forall v\in T_{x}\widetilde{N}\}.
\label{1y}%
\end{equation}

Let $tr(T\widetilde{N})$ and $ltr(T\widetilde{N})$ be complementary (but not
orthogonal) vector bundles to $T\widetilde{N}$ in $T\overline{N}%
|_{\widetilde{N}}$ and to $Rad(T\widetilde{N})$ in $S(T\widetilde{N}^{\bot
})^{\bot},$ respectively. Then we get
\begin{equation}
tr(T\widetilde{N})=ltr(T\widetilde{N})\bot S(T\widetilde{N}^{\bot}) \label{2}%
\end{equation}
and
\begin{align}
T\overline{N}|_{\widetilde{N}}  &  =T\widetilde{N}\oplus tr(T\widetilde
{N})\nonumber\\
&  =(Rad(T\widetilde{N})\oplus ltr(T\widetilde{N}))\bot S(T\widetilde{N})\bot
S(T\widetilde{N}^{\bot}). \label{3}%
\end{align}

Here, $ltr(T\widetilde{N})$ is called lightlike transversal vector bundle in
$T\overline{N}$ such that there exists a local null frame $\{N_{i}\}$ of
sections with values in the orthogonal complement of $S(T\widetilde{N}^{\bot
})$ in $[S(T\widetilde{N})]^{\bot}$ and $ltr(T\widetilde{N})$ locally spanned
by $\{N_{i}\}.$ Then, take into account any local basis $\{\xi_{i}\}$ of
$Rad(T\widetilde{N})$ and $\{N_{i}\}$ of $ltr(T\widetilde{N}),$ we have
$\overline{\rho}(\xi_{i},N_{j})=\delta_{ij}$ and $\overline{\rho}(N_{i}%
,N_{j})=0,$ $i,j=1,...,s.$

For a lightlike submanifold $(\widetilde{N},\rho,S(T\widetilde{N}%
),S(T\widetilde{N}^{\bot}))$,\ there are the following four cases:

\textit{(i)} $\widetilde{N}$ is called s-lightlike, if $s<\min\{p,q\},$

\textit{(ii)} $\widetilde{N}$ is called co-isotropic, if $s=q<p$ , that is,
$S(T\widetilde{N}^{\bot})=\{0\},$

\textit{(iii)} $\widetilde{N}$ is called isotropic, if $s=p<q$, that is,
$S(T\widetilde{N})=\{0\},$

\textit{(iv)} $\widetilde{N}$ is called totally lightlike, if $s=p=q,$ that
is, $S(T\widetilde{N})=\{0\}=S(T\widetilde{N}^{\bot}).$

Now, let $\overline{\nabla}$ be the Levi-Civita connection on $\overline{N}.$
According to (\ref{3}), for $\forall X,Y\in\Gamma(T\widetilde{N})$ and
$\forall V\in\Gamma(tr(T\widetilde{N})),$ the Gauss and Weingarten formula of
$\widetilde{N}$ are given by%
\begin{equation}
\overline{\nabla}_{X}Y=\nabla_{X}Y+h(X,Y) \label{4}%
\end{equation}
and%

\begin{equation}
\overline{\nabla}_{X}V=-A_{V}X+\nabla_{X}^{t}V, \label{5}%
\end{equation}
where $\{\nabla_{X}Y,A_{V}X\}$ and $\{h(X,Y),\nabla_{X}^{t}V\}$ are belong to
$\Gamma(T\widetilde{N})$ and $\Gamma(tr(T\widetilde{N})),$ respectively. Also,
$\nabla$ and $\nabla^{t}$ are linear connections on $\widetilde{N}$ and on the
vector bundle $tr(T\widetilde{N})$, respectively. The second fundamental form
$h$ is a symmetric $\mathcal{F}(\widetilde{N})$-bilinear form on
$\Gamma(T\widetilde{N})$ with values in $\Gamma(tr(T\widetilde{N}))$ and the
shape operator $A_{V}$ is a linear endomorphism of $\Gamma(T\widetilde{N})$.
Then, for $\forall X,Y\in\Gamma(T\widetilde{N}),$ $N\in\Gamma(ltr(T\widetilde
{N}))$ and $Z\in\Gamma(S(T\widetilde{N}^{\bot})),$ we get%
\begin{align}
\overline{\nabla}_{X}Y  &  =\nabla_{X}Y+h^{l}(X,Y)+h^{s}(X,Y),\label{6}\\
\overline{\nabla}_{X}N  &  =-A_{N}X+\nabla_{X}^{l}(N)+D^{s}(X,N),\label{6'}\\
\overline{\nabla}_{X}Z  &  =-A_{Z}X+D^{l}(X,Z)+\nabla_{X}^{s}(Z), \label{6''}%
\end{align}
where $\left\{  \nabla_{X}^{l}(N),D^{l}(X,Z)\right\}  $ and $\left\{
D^{s}(X,N),\nabla_{X}^{s}(Z)\right\}  $ are parts of $ltr(T\widetilde{N})$ and
$S(T\widetilde{N}^{\perp}),$ respectively. Also, $h^{l}(X,Y)$ $=Lh(X,Y)\in
\Gamma(ltr(T\widetilde{N}))$ and $h^{s}(X,Y)=Sh(X,Y)\in\Gamma(S(T\widetilde
{N}^{\perp})),$ where $L$ and $S$ are the projectors of transversal vector
bundle $tr(T\widetilde{N})$ on $ltr(T\widetilde{N})$ and $S(T\widetilde
{N}^{\perp}).$ Denote the projection of $T\widetilde{N}$ on $S(T\widetilde
{N})$ by $Q$. Thus, using (\ref{4}), (\ref{6}), (\ref{6''}) and considering
the metric connection $\overline{\nabla},$ we have%
\begin{align}
\overline{\rho}(h^{s}(X,Y),Z)+\overline{\rho}(Y,D^{l}(X,Z))  &  =\rho
(A_{Z}X,Y),\label{7}\\
\overline{\rho}(D^{s}(X,N),Z)  &  =\overline{\rho}(N,A_{Z}X) \label{7'}%
\end{align}
and%
\begin{align}
\nabla_{X}QY  &  =\nabla_{X}^{\ast}QY+h^{\ast}(X,QY),\text{ }\label{7''}\\
\nabla_{X}\xi &  =-A_{\xi}^{\ast}X+\nabla_{X}^{\ast^{t}}\xi, \label{7'''}%
\end{align}
where $\forall X,Y\in\Gamma(T\widetilde{N})$, $\xi\in\Gamma(Rad(T\widetilde
{N}))$ and $\nabla^{\ast}$ is induced connection on $S(T\widetilde{N})$ which
is a metric connection and $\nabla^{\ast^{t}}$ is induced connection on
$Rad(T\widetilde{N}).$ Using (\ref{7}), (\ref{7'}), (\ref{7''}) and
(\ref{7'''}), we get%
\begin{align}
\overline{\rho}(h^{l}(X,QY),\xi)  &  =\rho(A_{\xi}^{\ast}X,QY),\text{
\ \ }\overline{\rho}(h^{\ast}(X,QY),N)=\rho(A_{N}X,QY),\label{8}\\
\overline{\rho}(h^{l}(X,\xi),\xi)  &  =0\text{ and }A_{\xi}^{\ast}\xi=0.
\label{8'}%
\end{align}

Also, the induced connection on lightlike submanifold $\widetilde{N}$ is
torsion-free, but it is not metric connection and satisfies the following
condition%
\begin{equation}
(\nabla_{X}\rho)(Y,Z)=\overline{\rho}(h^{l}(X,Y),Z)+\overline{\rho}%
(h^{l}(X,Z),Y). \label{9}%
\end{equation}

For more details and studies about lightlike submanifolds, one can see
\cite{Bejancu}, \cite{Bayrm2}, \cite{K1}, \cite{K2}, \cite{K3}, \cite{Sahin1},
\cite{Sahin2}, \cite{Sahin3}, \cite{Sahin4}, \cite{Sahin5}.

On the other hand, an indefinite almost Hermitian manifold $(\overline
{N},\overline{\rho},\overline{J})$ is a $2m$-dimensional semi-Riemannian
manifold $\overline{N}$ with a semi-Riemannian metric $\rho$ of the constant
index $q$, $0<q<2m$ and a $(1,1)$ tensor field $\overline{J}$ on $\overline
{N}$ such that for $\forall X,Y\in\Gamma(T\overline{N})$ the following
conditions are satisfied%
\begin{align}
\overline{J}^{2}X  &  =-X,\label{10}\\
\overline{\rho}(\overline{J}X,\overline{J}Y)  &  =\overline{\rho}(X,Y).
\label{10'}%
\end{align}

If $\overline{J}$ is parallel with respect to $\overline{\nabla}$, that is,
\begin{equation}
(\overline{\nabla}_{X}\overline{J})Y=0, \label{11}%
\end{equation}
then an indefinite almost Hermitian manifold $(\overline{N},\overline{\rho
},\overline{J})$ is called an indefinite Kaehler manifold \cite{Barros}, where
$\forall X,Y\in\Gamma(T\overline{N})$ and $\overline{\nabla}$ is the
Levi-Civita connection with respect to $\overline{\rho}$.

\section{ \textbf{SCREEN ALMOST SEMI-INVARIANT LIGHTLIKE SUBMANIFOLDS OF
INDEFINITE KAEHLER MANIFOLDS}}

In this section, we introduce screen almost semi-invariant (SASI) lightlike
submanifolds, give an example and obtain some characterizations. We give the
necessary and sufficient condition for the induced connection which is not
metric connection in general. Finally, we investigate the notion of mixed
geodesic for (SASI) lightlike submanifolds.

\begin{definition}
Let $(\widetilde{N},\rho,S(T\widetilde{N}))$ be a lightlike submanifold of an
indefinite Kaehler manifold $(\overline{N},\overline{\rho},\overline{J})$.
Then we say that $\widetilde{N}$ is a SASI-lightlike submanifold of
$\overline{N},$ if the following conditions are satisfied:

\textbf{i)} $D_{0},D_{1},D$ and $D^{\ast}$ are orthogonal distributions on
$S(T\widetilde{N})$ such that
\[
S(T\widetilde{N})=D\oplus_{\perp}D^{\ast}\oplus_{\perp}D^{^{\prime\prime}},
\]
where $\overline{J}D=D$, $D^{\ast}$ anti-invariant and $D^{^{\prime\prime}%
}=D_{0}\oplus_{\perp}D_{1}$. Also $\overline{J}D^{^{\prime\prime}}\nsubseteq
S(T\widetilde{N})$ and $\overline{J}D^{^{\prime\prime}}\nsubseteq
S(T\widetilde{N}^{\perp})$ such that the ditribution $D^{^{\prime\prime}}$ is
neither invariant nor anti-invariant.

\textbf{ii) }$\overline{J}(Rad(T\widetilde{N}))=Rad(T\widetilde{N})$, that is
$Rad(T\widetilde{N})$ is invariant respect to $\overline{J}$.

\textbf{iii) }$S(T\widetilde{N}^{\perp})=$ $\overline{J}D^{\ast}\oplus_{\perp
}\mu$, $\mu=D_{2}\oplus_{\perp}D_{3}$, $\overline{J}\mu\neq\mu$.
\end{definition}

Also, we have $\overline{J}ltr(T\widetilde{N})=ltr(T\widetilde{N})$ and the
following decomposition%
\[
T\widetilde{N}=D^{^{\prime}}\oplus_{\perp}D^{\ast}\oplus_{\perp}%
D^{^{\prime\prime}},
\]
where $D^{^{\prime}}=D\oplus_{\perp}Rad(T\widetilde{N})$ and $D^{^{\prime}}$
is invariant with respect to $\overline{J}$.

We denote the orthogonal complement to $\overline{J}D^{\ast}$ in
$S(T\widetilde{N}^{\perp})$ by $\mu$. Then, we have%
\[
tr(T\widetilde{N})=ltr(T\widetilde{N})\oplus_{\perp}\overline{J}D^{\ast}%
\oplus_{\perp}\mu.
\]

\begin{proposition}
A SASI-lightlike submanifold $\widetilde{N}$ of an indefinite Kaehler manifold
$\overline{N}$ is a SSI-lightlike submanifold if and only if $D_{0}$ and
$D_{1}$ (similarly $D_{2}$\ and $D_{3}$) are invariant distributions.
\end{proposition}

\begin{proof}
Let $\widetilde{N}$ be a SSI-lightlike submanifold of $\overline{N}.$ Then
form \cite{Erol}, we have $S(T\widetilde{N})=D\bot D^{\ast},$ $\overline
{J}D=D,$ $\overline{J}D^{\ast}\subset S(T\widetilde{N}^{\bot})$ and\ $D\cap
D^{\ast}=\{0\}.$ In this case, either $D_{0}$ and $D_{1}$ are zero or $D_{0}$
and $D_{1}$ are invariant. If $D_{0}$ and $D_{1}$ are zero, then $\mu$ is
zero. But, since $\widetilde{N}$ is the SSI-lightlike submanifold, $\mu$
cannot be zero. Therefore, $D_{0}$ and $D_{1}$ must be invariant. This imply
that $\mu=D_{2}\oplus_{\bot}D_{3}$ is invariant. Conversely, assume that
$D_{0}$ and $D_{1}$ (similarly $D_{2}$\ and $D_{3}$) are invariant
distributions. Then, we get $D^{^{\prime}}=D\oplus_{\bot}Rad(T\widetilde
{N})\oplus_{\bot}D_{0}\oplus_{\bot}D_{1}.$ Thus, we obtain%
\[
T\widetilde{N}=D^{^{\prime}}\oplus_{\bot}D^{\ast}.
\]

Since we have $\mu=D_{2}\oplus_{\bot}D_{3}$, $\mu\subset S(T\widetilde
{N}^{\bot}).$ Here, we get $S(T\widetilde{N})=\widetilde{D}\bot D^{\ast},$
where $\widetilde{D}=D\oplus_{\bot}D_{0}\oplus_{\bot}D_{1}.$ Then,
$\widetilde{N}$ is a SSI-lightlike submanifold of $\overline{N}.$ Thus, the
proof completes.
\end{proof}

Now, we will constant an example of SASI-lightlike submanifold in $%
\mathbb{R}
_{2}^{12}.$

\begin{example}
Let $(%
\mathbb{R}
_{2}^{12},\overline{\rho},\overline{J})$ be an indefinite Kaehler manifold
with signature $(-,-,+,+,+,\newline+,+,+,+,+,+,+)$ and let $\widetilde{N}$ be
a submanifold of $%
\mathbb{R}
_{2}^{12}$ given by%
\begin{align*}
F(u,v,t,m,n,l,w)  &  =(u,v,-\cos t,\sin t,-t\cos m,-t\sin m,n\cos\beta
-w\sin\beta,\\
&  n\sin\beta+w\cos\beta,\cos l,\sin l,u\cos\alpha-v\sin\alpha,u\sin
\alpha+v\cos\alpha).
\end{align*}

Then we have $T\widetilde{N}=Sp\left\{  F_{1},F_{2},F_{3},F_{4},F_{5}%
,F_{6},F_{7}\right\}  $ such that%
\begin{align*}
F_{1}  &  =\partial x_{1}+\cos\alpha\partial x_{11}+\sin\alpha\partial
x_{12},\\
F_{2}  &  =\partial x_{2}-\sin\alpha\partial x_{11}+\cos\alpha\partial
x_{12},\\
F_{3}  &  =\sin t\partial x_{3}+\cos t\partial x_{4}-\cos m\partial x_{5}-\sin
m\partial x_{6},\\
F_{4}  &  =t\sin m\partial x_{5}-t\cos m\partial x_{6},\\
F_{5}  &  =\cos\beta\partial x_{7}+\sin\beta\partial x_{8},\\
F_{6}  &  =-\sin\beta\partial x_{7}+\cos\beta\partial x_{8},\\
F_{7}  &  =-\sin l\partial x_{9}+\cos l\partial x_{10}.
\end{align*}

Since $Rad(T\widetilde{N})=Sp\left\{  F_{1},F_{2}\right\}  $ and $\overline
{J}F_{1}=F_{2}$, one can easily see that\ $\widetilde{N}$ is a 2- lightlike
submanifold. Now, we have $D=Sp\left\{  F_{5},F_{6}\right\}  $ such that
$\overline{J}F_{5}=F_{6}$ which implies that $D$ is invariant with respect to
$\overline{J}$. Also, we get $D^{\ast}=Sp\left\{  F_{7}\right\}  $ and
$D^{^{\prime\prime}}=Sp\left\{  F_{3},F_{4}\right\}  $. On the other hand, we
obtaine the lightlike transversal bundle $ltr(T\widetilde{N})=Sp\left\{
N_{1},N_{2}\right\}  $, where%
\begin{align*}
N_{1}  &  =\frac{1}{2}\left\{  -\partial x_{1}+\cos\alpha\partial x_{11}%
+\sin\alpha\partial x_{12}\right\}  ,\\
N_{2}  &  =\frac{1}{2}\left\{  -\partial x_{2}-\sin\alpha\partial x_{11}%
+\cos\alpha\partial x_{12}\right\}
\end{align*}
and the screen transversal bundle $S(T\widetilde{N}^{\perp})=Sp\left\{
W_{1},W_{2},W_{3}\right\}  $, where%
\begin{align*}
W_{1}  &  =-\cos t\partial x_{3}+\sin t\partial x_{4},\text{ \ \ \ \ \ \ \ \ }%
W_{2}=-\cos l\partial x_{9}+\sin l\partial x_{10},\\
W_{3}  &  =\sin t\partial x_{3}+\cos t\partial x_{4}+\cos m\partial
x_{5}\text{+}\sin m\partial x_{6}\text{.\ \ \ }%
\end{align*}

Here we get $\overline{J}F_{7}=W_{2}$, $\overline{J}N_{1}=N_{2}$ and
$\mu=Sp\left\{  W_{1},W_{3}\right\}  .$ Then, since
\[
\overline{J}F_{3}=W_{1}+\frac{1}{t}F_{4},\ \ \overline{J}F_{4}=\frac{t}%
{2}W_{3}-\frac{t}{2}F_{3}%
\]
and
\[
\overline{J}W_{1}=-\frac{1}{2}F_{3}-\frac{1}{2}W_{3},\text{ \ \ }\overline
{J}W_{3}=W_{1}-\frac{1}{t}F_{4}%
\]
we obtain that $\widetilde{N}$ is a SASI-lightlike submanifold of \ $%
\mathbb{R}
_{2}^{12}.$
\end{example}

Now, let $P_{0},P_{1},P_{2},P_{3}$ and $S_{1}$ be the projections on
$D,D^{\ast},D_{0},D_{1}$ and $Rad(T\widetilde{N})$ in $T\widetilde{N}$,
respectively. Similarly, let $R_{1},R_{2}$ and $R_{3}$ be the the projection
on $\overline{J}D^{\ast},D_{2}$ and $D_{3\text{ }}$in $tr(T\widetilde{N})$,
respectively. Then for any $X\in\Gamma(T\widetilde{N}),$ we get%
\begin{align}
X  &  =PX+RX+QX\nonumber\\
&  =P_{0}X+P_{1}X+P_{2}X+P_{3}X+S_{1}X, \label{a}%
\end{align}
where $PX\in\Gamma(D^{^{\prime}}),$ $RX\in\Gamma(D^{^{\ast}})$ and
$QX\in\Gamma(D^{^{^{\prime\prime}}}).$ \bigskip

On the other hand, one can write that
\begin{equation}
\overline{J}X=TX+wX, \label{a!}%
\end{equation}
where $TX$ and $wX$ are the tangential and transversal components of
$\overline{J}X$, respectively. Then applying $\overline{J}$ to equation
(\ref{a}) and substituting $T_{0},w_{1},T_{2}+w_{2},T_{3}+w_{3},T_{1}$ for
$\overline{J}P_{0},\overline{J}P_{1},\newline\overline{J}P_{2},\overline
{J}P_{3},\overline{J}S_{1},$ respectively, we have%
\begin{equation}
\overline{J}X=T_{0}X+w_{1}X+T_{2}X+w_{2}X+T_{3}X+w_{3}X+T_{1}X, \label{b}%
\end{equation}
where it is clear that $T_{0}X\in\Gamma(\overline{J}D),$ $w_{1}X\in
\Gamma(\overline{J}D^{\ast}),$ $(T_{2}X+w_{2}X)\in\Gamma(\overline{J}%
D_{0}=D_{1}\oplus_{\perp}D_{2}),$ $(T_{3}X+w_{3}X)\in\Gamma(\overline{J}%
D_{1}=D_{0}\oplus_{\perp}D_{3}),$ $T_{1}X\in\Gamma(\overline{J}%
(Rad(T\widetilde{N}))).$

\bigskip

Similarly, for any $\Omega\in\Gamma(tr(T\widetilde{N})),$ it is known that
\begin{equation}
\overline{J}\Omega=B\Omega+C\Omega, \label{c}%
\end{equation}
where $B\Omega$ and $C\Omega$ are the sections of $T\widetilde{N}$ and
$tr(T\widetilde{N})$, respectively. Thus, differentiating (\ref{b}) and using
the equations Gauss-Weingarten in (\ref{6}-\ref{6''}) and (\ref{c}), we obtain
the following equations%
\begin{align}
(\bigtriangledown_{X}T)Y=A_{w_{1}Y}X+A_{w_{2}Y}X+A_{w_{3}Y}X+Bh(X,Y),\label{d}%
\\
D^{l}(X,w_{1}Y)+D^{l}(X,w_{2}Y)+D^{l}(X,w_{3}Y)=Ch^{l}(X,Y)-h^{l}(X,TY)
\label{e}%
\end{align}
and
\begin{equation}
h^{s}(X,TY)+\bigtriangledown_{X}^{s}w_{1}Y+\bigtriangledown_{X}^{s}%
w_{2}Y+\bigtriangledown_{X}^{s}w_{3}Y-w(\bigtriangledown_{X}Y)-Ch^{s}(X,Y)=0.
\label{f}%
\end{equation}

\begin{theorem}
Let $\widetilde{N}$ be a SASI-lightlike submanifold (not proper) of an
indefinite Kaehler manifold $\overline{N}.$ Then, we have the following statements

\textbf{i) }If\textbf{ }$\widetilde{N}$ is isotropic, coisotropic or totally
lightlike submanifold, then $\widetilde{N}$ is invariant,

\textbf{ii) }If\textbf{ }$\widetilde{N}$ is an isotropic lightlike
submanifold, then $\widetilde{N}$ is also a totally lightlike submanifold.
\end{theorem}

\begin{proof}
\textit{i)} Let $\widetilde{N}$ be a SASI-lightlike submanifold (not proper)
of an indefinite Kaehler manifold $\overline{N}.$ If $\widetilde{N}$ is
isotropic, then $S(T\widetilde{N})=\left\{  0\right\}  $ which implies that
$D=\left\{  0\right\}  ,$ $D^{\ast}=\left\{  0\right\}  $ and $D^{^{\prime
\prime}}=\left\{  0\right\}  .$ So, we have $T\widetilde{N}=Rad(T\widetilde
{N})=\overline{J}(Rad(T\widetilde{N})),$ which is invariant respect to
$\overline{J}$. If $\widetilde{N}$ is coisotropic, then $S(T\widetilde
{N}^{\perp})$ $=\left\{  0\right\}  $ implies $%
\mu
=\left\{  0\right\}  ,$ then $T\widetilde{N}=D^{^{\prime}}=D\oplus_{\perp
}Rad(T\widetilde{N})$ such that $\widetilde{N}$ is invariant. And if
$\widetilde{N}$ is totally lightlike, since both $S(T\widetilde{N})=\left\{
0\right\}  $ and $S(T\widetilde{N}^{\perp})$ $=\left\{  0\right\}  ,$ we get
$T\widetilde{N}=Rad(T\widetilde{N})=\overline{J}(Rad(T\widetilde{N})).$
Consequently $\widetilde{N}$ is invariant.

\textit{ii) }The proof is obvious form \textit{(i)}.
\end{proof}

\begin{theorem}
\label{A}Let $\widetilde{N}$ be a SASI-lightlike submanifold of an indefinite
Kaehler manifold $\overline{N}.$ The induced connection $\nabla$ is a metric
connection if and only if
\[
(A_{\overline{J}Y}^{\ast}X-\bigtriangledown_{X}^{\ast^{t}}\overline{J}%
Y)\in\Gamma(Rad(T\widetilde{N}))\text{ \ \ \ and \ \ \ }Bh(X,\overline
{J}Y)=0,
\]
where $X\in\Gamma(T\widetilde{N})$ and $Y\in\Gamma(Rad(T\widetilde{N})).$
\end{theorem}

\begin{proof}
For almost complex structure $\overline{J},$ one can write from (\ref{11})%
\[
\overline{\bigtriangledown}_{X}Y=-\overline{J}(\overline{\bigtriangledown}%
_{X}\overline{J}Y),
\]
where $X,Y\in\Gamma(T\overline{N}).$ Then, using (\ref{6}) \ and (\ref{7'''}),
for $X\in\Gamma(T\widetilde{N})$ and $Y\in\Gamma(Rad(T\widetilde{N})),$ we
have
\begin{align}
\overline{\bigtriangledown}_{X}Y  &  =-\overline{J}(A_{\overline{J}Y}^{\ast
}X-\bigtriangledown_{X}^{\ast^{t}}\overline{J}Y)-Bh(X,\overline{J}%
Y)-Ch(X,\overline{J}Y)\nonumber\\
&  =TA_{\overline{J}Y}^{\ast}X+wA_{\overline{J}Y}^{\ast}X-T\bigtriangledown
_{X}^{\ast^{t}}\overline{J}Y-w\bigtriangledown_{X}^{\ast^{t}}\overline
{J}Y-Bh(X,\overline{J}Y)-Ch(X,\overline{J}Y). \label{g}%
\end{align}

Taking the tangential parts of (\ref{g}), we get
\begin{equation}
\bigtriangledown_{X}Y=TA_{\overline{J}Y}^{\ast}X-T\bigtriangledown_{X}%
^{\ast^{t}}\overline{J}Y-Bh(X,\overline{J}Y). \label{h}%
\end{equation}

According to (\ref{h}), if $T(A_{\overline{J}Y}^{\ast}X-\bigtriangledown
_{X}^{\ast^{t}}\overline{J}Y)\in\Gamma(Rad(T\widetilde{N}))$ and
$Bh(X,\overline{J}Y)=0$, then $\bigtriangledown_{X}Y\in\Gamma(Rad(T\widetilde
{N})).$ Consequently, the proof is completed.
\end{proof}

\begin{theorem}
Let $\widetilde{N}$ be a SASI-lightlike submanifold of an indefinite Kaehler
manifold $\overline{N}.$ Then distribution $D^{^{\prime}}$ is integrable if
and only if
\[
h(X,\overline{J}Y)=h(\overline{J}X,Y)
\]
holds for $\forall X,Y\in\Gamma(D^{^{\prime}})$.
\end{theorem}

\begin{proof}
$D^{^{\prime}}$ is integrable if and only if\ for $\forall X,Y\in
\Gamma(D^{^{\prime}}),$ $[X,Y]\in\Gamma(D^{^{\prime}}).$ Here for $\forall
X,Y\in\Gamma(D^{^{\prime}}),$ obviously $wX=wY=0.$ From the equations
(\ref{e}) and (\ref{f}), we obtain
\[
-Ch^{l}(X,Y)+h^{l}(X,TY)+h^{s}(X,TY)-w(\bigtriangledown_{X}Y)-Ch^{s}(X,Y)=0.
\]

Thus, we get%
\begin{equation}
w(\bigtriangledown_{X}Y)=h(X,TY)-Ch(X,Y). \label{i}%
\end{equation}

Now, using (\ref{i}) and symmetry property of the second fundamental form $h,$
we have $w([X,Y])=h(X,TY)-h(TX,Y).$ Then the proof completes.
\end{proof}

\begin{theorem}
Let $\widetilde{N}$ be a SASI-lightlike submanifold of an indefinite Kaehler
manifold $\overline{N}.$ Then distribution $D^{^{^{\ast}}}$ is integrable if
and only if for $\forall V,W\in\Gamma(D^{^{^{\ast}}})$ \
\[
A_{\overline{J}V}W=A_{\overline{J}W}V.
\]

\end{theorem}

\begin{proof}
Since $TV=TW=0$ and from (\ref{d}), we obtain $-T\nabla_{V}W=A_{wW}V+Bh(V,W)$
for $\forall V,W\in\Gamma(D^{^{^{\ast}}})$. Then, we have $T([V,W])=A_{wW}%
V-A_{wV}W.$ Thus, the proof completes.
\end{proof}

\begin{theorem}
Let $\widetilde{N}$ be a SASI-lightlike submanifold of an indefinite Kaehler
manifold $\overline{N}.$ Then, the distribution $D^{^{^{\prime\prime}}}$ is
integrable if and only if the following properties hold for $\forall
Z,U\in\Gamma(D^{^{^{\prime\prime}}})$

\textbf{i)} $(h(Z,TU)-h(U,TZ)+\bigtriangledown_{U}^{s}wZ-\bigtriangledown
_{Z}^{s}wU)\in\Gamma(\mu)$ and $D^{l}(U,wZ)=D^{l}(Z,wU),$

\textbf{ii)} $(\nabla_{Z}TU-\nabla_{U}TZ-A_{wU}Z+A_{wZ}U)\in\Gamma
(D^{^{\prime\prime}}).$
\end{theorem}

\begin{proof}
Since $D^{^{\prime\prime}}$ is neither an invariant nor an anti-invariant
distribution on $\widetilde{N},$ we have $TZ\neq0\neq TU$ and $wZ\neq0\neq
wU.$ Using (\ref{d})-(\ref{e})-(\ref{f}), we obtaine
\begin{align*}
\nabla_{Z}TU+h(Z,TU)+\left\{  -A_{wU}Z+\bigtriangledown_{U}^{s}wZ+D^{l}%
(Z,wU)\right\}   &  =T\nabla_{Z}U+w(\nabla_{Z}U)\\
&  +Bh(Z,U)+Ch(Z,U).
\end{align*}

Then, using symmetry property of the second fundamental form $h,$ we complete
the proof.
\end{proof}

\bigskip Now, we can examine the parallelism of distributions $D^{^{\prime}},$
$D^{\ast}$ and $D^{^{\prime\prime}}$on $\widetilde{N}.$

\begin{theorem}
\label{B}Let $\widetilde{N}$ be a SASI-lightlike submanifold of an indefinite
Kaehler manifold $\overline{N}.$Then, the distribution $D^{^{\prime}}$ is
parallel if and only if the following properties hold for $\forall
X,Y\in\Gamma(D^{^{\prime}}),$ $\forall V\in\Gamma(D^{^{^{\ast}}})$ and
$\forall Z\in\Gamma(D^{^{^{\prime\prime}}})$

\textbf{i)} For $Y\in\Gamma(D),$ $A_{\overline{J}V}X$ and $\nabla_{X}%
TZ-A_{wZ}X$ have no component on $\Gamma(D).$

\textbf{ii)} For $Y\in\Gamma(Rad(T\widetilde{N})),$ $h^{l}(X,TZ)=-D^{l}(X,wZ)$
and $\overline{\rho}(D^{l}(X,wV),\overline{J}Y)=0.$
\end{theorem}

\begin{proof}
Assume that $D^{^{\prime}}$ is parallel distribution. Then for $\forall
X,Y\in\Gamma(D^{^{\prime}}),$ $\forall V\in\Gamma(D^{^{^{\ast}}})$ and
$\forall Z\in\Gamma(D^{^{^{\prime\prime}}}),$ we get
\[
\rho(\nabla_{X}Y,V)=0=\rho(\nabla_{X}Y,Z).
\]

Using $\overline{\nabla}$ is a metric connection, from (\ref{6}), (\ref{6''})
and (\ref{a!}), we have%
\begin{equation}
\rho(\nabla_{X}Y,V)=-\overline{\rho}(-A_{wV}X+D^{l}(X,wV)+\bigtriangledown
_{X}^{s}wV,\overline{J}Y). \label{k}%
\end{equation}

On the other hand, as a result of the necessary calculations, we get%
\begin{align}
\rho(\nabla_{X}Y,Z)  &  =-\overline{\rho}(\nabla_{X}TZ+h^{l}(X,TZ)+h^{s}%
(X,TZ)\nonumber\\
&  -A_{wZ}X+D^{l}(X,wZ)+\bigtriangledown_{X}^{s}wZ,\overline{J}Y). \label{l}%
\end{align}

Then, form (\ref{k}) and (\ref{l}), we prove that for $Y\in\Gamma(D)$ and
$Y\in\Gamma(Rad(T\widetilde{N}))$ the conditions \textit{(i) }and\textit{
(ii)} are true. The converse of the proof of is obvious.
\end{proof}

\begin{theorem}
\label{C}Let $\widetilde{N}$ be a SASI-lightlike submanifold of an indefinite
Kaehler manifold $\overline{N}.$ Then, the distribution $D^{^{\ast}}$ is
parallel if and only if

\textbf{i)} For $\forall V,W\in\Gamma(D^{^{^{\ast}}}),$ $A_{wV}W$ has no
components on $\Gamma(D)$ , $\Gamma(D^{^{\prime\prime}})$ and $\Gamma
(Rad(T\widetilde{N})),$

\textbf{ii)} For $Z\in\Gamma(D^{^{^{\prime\prime}}}),$ $\bigtriangledown
_{X}^{s}wV$ has no components on $\Gamma(\mu).$
\end{theorem}

\begin{proof}
Assume that $D^{^{\ast}}$ is parallel distribution. Then for $\forall
X\in\Gamma(D),$ $\forall V,W\in\Gamma(D^{^{^{\ast}}})$ and $N\in
\Gamma(ltr(T\widetilde{N})),$ we get
\[
\rho(\nabla_{W}V,X)=\overline{\rho}(\nabla_{W}V,N)=\rho(\nabla_{W}V,Z)=0.
\]

Then using (\ref{10'}), (\ref{6}) and (\ref{6''}), we obtain
\begin{equation}
g(\nabla_{W}V,X)=\overline{g}(-A_{wV}W+D^{l}(W,wV)+\bigtriangledown_{W}%
^{s}wV,\overline{J}X). \label{m}%
\end{equation}

Also, it is easy to see that
\begin{equation}
\rho(\nabla_{W}V,N)=\overline{\rho}(-A_{wV}W+D^{l}(W,wV)+\bigtriangledown
_{W}^{s}wV,\overline{J}N) \label{n}%
\end{equation}
and
\begin{equation}
\rho(\nabla_{W}V,Z)=\overline{\rho}(-A_{wV}W+\bigtriangledown_{W}%
^{s}wV,\overline{J}Z). \label{n'}%
\end{equation}

Thus, taking into account the equations (\ref{m}) and (\ref{n}), the condition
(i) of the proof is complete. From (\ref{n'}), (ii) is obvious.
\end{proof}

\begin{theorem}
\label{D}Let $\widetilde{N}$ be a SASI-lightlike submanifold of an indefinite
Kaehler manifold $\overline{N}.$ Then, the distribution $D^{^{^{\prime\prime}%
}}$is parallel if and only if the following conditions are satisfied

\textbf{i)} $\nabla_{Z}TU-A_{wU}Z$ has no components on both $\Gamma(D)$ and
$\Gamma(Rad(T\widetilde{N})),$

\textbf{ii)} $h^{s}(Z,TU)=-\nabla_{Z}^{s}wU,$ where $U,Z\in\Gamma
(D^{^{^{\prime\prime}}}).$
\end{theorem}

\begin{proof}
Assume that $D^{^{^{\prime\prime}}}$is parallel distribution. Then, for
$\forall X\in\Gamma(D),$ $\forall Z,U\in\Gamma(D^{^{^{\prime\prime}}})$ and
$N\in\Gamma(ltr(T\widetilde{N})),$ we get%
\[
\rho(\nabla_{Z}U,X)=\rho(\nabla_{Z}U,V)=\overline{\rho}(\nabla_{Z}U,N)=0.
\]

Then using (\ref{10'}), (\ref{6}) and (\ref{6''}), we obtain%
\begin{equation}
\rho(\nabla_{Z}U,X)=\overline{\rho}(\nabla_{Z}TU-A_{wU}Z,\overline{J}X).
\label{p}%
\end{equation}

Now, with easy calculations, we have
\begin{equation}
\rho(\nabla_{Z}U,N)=\overline{\rho}(\nabla_{Z}TU-A_{wU}Z,\overline{J}N).
\label{q}%
\end{equation}

Hence, the equations (\ref{p}) and (\ref{q}) prove \textit{(i)}.

On the other hand, for $\forall V\in\Gamma(D^{^{^{\ast}}}),$ we get%
\begin{equation}
\rho(\nabla_{Z}U,V)=\overline{\rho}(h^{s}(Z,TU)+\nabla_{Z}^{s}wU,\overline
{J}V) \label{o}%
\end{equation}
and this proves \textit{(ii)}.
\end{proof}

\begin{definition}
Let $\widetilde{N}$ be a SASI-lightlike submanifold of an indefinite Kaehler
manifold $\overline{N}.$ If its second fundamental form $h$ satisfies%
\begin{equation}
h(X,Y)=0,\text{ \ \ }\forall X,Y\in\Gamma(D), \label{o'}%
\end{equation}
then $\widetilde{N}$ is called a $D$- geodesic. Obviously, $\widetilde{N}$ is
a $D$-geodesic SASI-lightlike submanifold if
\begin{equation}
h^{l}(X,Y)=h^{s}(X,Y)=0,\text{ \ \ }\forall X,Y\in\Gamma(D).\text{\ }
\label{o''}%
\end{equation}

Also, if $h$ satisfies%
\begin{equation}
h(X,Y)=0,\text{ \ \ }\forall X\in\Gamma(D),\text{ \ }\forall Y\in
\Gamma(D^{^{\prime}})\text{ }(\text{or }\Gamma(D^{^{\prime\prime}})),
\label{o'''}%
\end{equation}
then $\widetilde{N}$ is called a mixed geodesic SASI-lightlike submanifold.
\end{definition}

\begin{theorem}
\label{E}Let $\widetilde{N}$ be a SASI-lightlike submanifold of an indefinite
Kaehler manifold $\overline{N}.$ Then, the distribution $D^{^{\prime}}$ of
$\widetilde{N}$ is a totally geodesic foliation in $\overline{N}$ if and only
if the following conditions are satisfied

\textbf{i)} $\widetilde{N}$ is $D^{^{\prime}}$-geodesic,

\textbf{ii)} $D^{^{\prime}}$ is parallel respect to $\nabla$ on $\widetilde
{N}$.
\end{theorem}

\begin{proof}
If $D^{^{\prime}}$ is a totally geodesic foliation in $\overline{N},$ then for
$\forall X,Y\in\Gamma(D^{^{\prime}}),$ we get$\ \overline{\nabla}_{X}%
Y\in\Gamma(D^{^{\prime}}).$ For $\forall\xi\in\Gamma(Rad(T\widetilde{N})),$
$V\in\Gamma(D^{^{^{\ast}}}),$ $Z\in\Gamma(D^{^{^{\prime\prime}}})$ and
$E\in\Gamma(S(T\widetilde{N}^{\bot})),$ we have%
\[
\overline{\rho}(\overline{\nabla}_{X}Y,\xi)=\overline{\rho}(\overline{\nabla
}_{X}Y,E)=\overline{\rho}(\overline{\nabla}_{X}Y,V)=\overline{\rho}%
(\overline{\nabla}_{X}Y,Z)=0.
\]

Then we obtain the following results:%
\begin{align}
\overline{\rho}(\overline{\nabla}_{X}Y,\xi)  &  =\overline{\rho}%
(h^{l}(X,Y),\xi),\label{s}\\
\overline{\rho}(\overline{\nabla}_{X}Y,E)  &  =\overline{\rho}(h^{s}%
(X,Y),E),\label{t}\\
\overline{\rho}(\overline{\nabla}_{X}Y,V)  &  =0,\label{u}\\
\overline{\rho}(\overline{\nabla}_{X}Y,Z)  &  =0. \label{v}%
\end{align}

From (\ref{s}) and (\ref{t}), we can write that
\[
h^{l}(X,Y)=0\text{\ \ \ and \ }h^{s}(X,Y)=0,\text{\ }for\forall X,Y\in
\Gamma(D^{^{\prime}}).
\]

Thus, it is clear that \ $\widetilde{N}$ is $D^{^{\prime}}$-geodesic, which
implies that $D^{^{\prime}}$ is parallel respect to $\nabla$ on $\widetilde
{N}$. Then, the \textit{(i)} and \textit{(ii)} conditions are proved.

The converse of proof is clear.
\end{proof}

\begin{theorem}
\label{F}Let $\widetilde{N}$ be a SASI-lightlike submanifold of an indefinite
Kaehler manifold $\overline{N}.$ Then $\widetilde{N}$ is mixed geodesic with
respect to $D^{\ast}$ if and only if the following assertions hold:

\textbf{i)} $\overline{\rho}(D^{l}(X,wV),\xi)=0,$

\textbf{ii)} For $E\in\Gamma(\overline{J}D^{\ast}),$ $\overline{\rho}%
(\nabla_{X}^{s}wV,wE)=0$ and for $E\in\Gamma(wD^{^{\prime\prime}})$,
$\rho(A_{wV}X,TE)=\overline{\rho}(\nabla_{X}^{s}wV,wE),$\newline where
$\forall X\in\Gamma(D^{^{\prime}}),$ $V\in\Gamma(D^{^{^{\ast}}}),$ $\xi
\in\Gamma(Rad(T\widetilde{N})),$ $E\in\Gamma(S(T\widetilde{N}^{\bot})).$
\end{theorem}

\begin{proof}
Assume that $\widetilde{N}$ is mixed geodesic. From (\ref{o''}) for $\forall
X\in\Gamma(D^{^{\prime}}),$ $V\in\Gamma(D^{^{^{\ast}}}),$ $\xi\in
\Gamma(Rad(T\widetilde{N})),$ we get $\overline{\rho}(h^{l}(X,V),\xi)=0.$
Then, we have $\overline{\rho}(\overline{\nabla}_{X}V,\xi)=0.$ Here,
taking$\overline{\text{ }J}\xi=\xi,$ we get $\overline{\rho}(\overline{\nabla
}_{X}V,\overline{J}\xi)=0.$ Thus, since $\overline{J}\overline{\nabla}%
_{X}\overline{J}V=-\overline{\nabla}_{X}\overline{J}V$ and (\ref{a!}), we
obtain\newline$\overline{\rho}(D^{l}(X,wV),\xi)=0.$

On the other hand, since $\widetilde{N}$ is mixed geodesic, we get
$\overline{\rho}(h^{s}(X,V),E)=0,$ where $\forall X\in\Gamma(D^{^{\prime}}),$
$V\in\Gamma(D^{^{^{\ast}}}),$ $E\in\Gamma(S(T\widetilde{N}^{\bot})).$ Then, we
have $\overline{\rho}(\overline{\nabla}_{X}V,E)=0.$ Similarly, since
$\overline{J}\overline{\nabla}_{X}\overline{J}V=-\overline{\nabla}%
_{X}\overline{J}V$ and (\ref{a!}), we obtain%
\begin{equation}
\rho(A_{wV}X-\nabla_{X}TV,TE)=\overline{\rho}(\nabla_{X}^{s}wV,wE). \label{y}%
\end{equation}

So, for $E\in\Gamma(\overline{J}D^{\ast})$ and $E\in\Gamma(wD^{^{\prime\prime
}}),$ we have our assertion.
\end{proof}

\begin{theorem}
\label{G}Let $\widetilde{N}$ be a SASI-lightlike submanifold of an indefinite
Kaehler manifold $\overline{N}.$ Then $\widetilde{N}$ is mixed geodesic with
respect to $D^{^{\prime\prime}}$ if and only if the following assertions hold:

\textbf{i)} $\overline{\rho}(D^{l}(X,wZ),\xi)=0,$

\textbf{ii)} For $E\in\Gamma(\overline{J}D^{\ast}),$ $\overline{\rho}%
(\nabla_{X}TZ-A_{wZ}X,TE)=0$ and for $E\in\Gamma(wD^{^{\prime\prime}})$,
$\overline{\rho}(A_{wZ}X-\nabla_{X}TZ,TE)=\overline{\rho}(\nabla_{X}%
^{s}wZ,wE),$\newline where $\forall X\in\Gamma(D^{^{\prime}}),$ $Z\in
\Gamma(D^{^{\prime\prime}}),$ $\xi\in\Gamma(Rad(T\widetilde{N})),$ $E\in
\Gamma(S(T\widetilde{N}^{\bot})).$
\end{theorem}

\begin{proof}
Getting $Z$ instead of $V$ in the previous theorem and its proof, one can
prove the theorem.
\end{proof}

\bigskip

\textbf{CONCLUDING REMARK.} \ One can recall that a lightlike hypersurface $M$
of a semi-Riemannian manifolds is totally geodesic if only if $Rad(TM)$ is a
Killing distribution on $M$ \cite{Bejancu}. In our study, Teorem \ref{E},
Teorem \ref{F} and Teorem \ref{G} of Section 3 are important in this sense.
Also, Teorem \ref{A}, Teorem \ref{B}, Teorem \ref{C} and Teorem \ref{D} are
about the screen shape operator or the second fundamental form which are
related to the notion of Killing horizon of general relativity \cite{Bayrm3}.
Indeed, "\textit{Solutions of the highly non-linear Einstein's field equations
require the assumption that they admit Killing or homothetic vector fields.
This is due to the fact that the Killing symmetries leave invariant the metric
connection, all the curvature quantities and the matter tensor of the Einstein
field equations of a spacetime. Most explicit solutions (see \cite{Bayrm} pp.
267) have been found by using one or more Killing vector fields. Related to
the theme of this section, here we show that the Killing symmetry has an
important role in the most active area of research on Killing horizons in
general relativity} \cite{Bayrm}." And it is known that the theory of
lightlike geometry has an important relation with general relativity.


\bigskip

\begin{thebibliography}{99}                                                                                               %


\bibitem {Atckn}Atceken, M. and K\i l\i\c{c}, E., Semi-Invariant Lightlike
Submanifolds of a Semi-Riemannian Product Manifold, Kodai Math. J., Vol. 30,
No. 3, pp. 361-378, 2007.

\bibitem {Oguz}Bahad\i r, O. Screen Semi-Invariant Half-Lightlike Submanifolds
of a Semi-Riemannian Product Manifold, Global Journal of Advanced Research on
Classical and Modern Geometries ISSN: 2284-5569, Vol.4, Issue 2, pp.116-124, 2015.

\bibitem {Barros}Barros, M. and Romero, A., Indefinite Kaehler manifolds,
Math. Ann., 261, 55--62, 1982.

\bibitem {Bejan}Bejan, C., Almost Semi-Invariant Submanifolds of Locally
Product Riemannian Manifolds, Bull. Math, de la Soc. Sci. Math, de la R.S. de
Roumanie Tome 32 (80), nr. 1, 1988.

\bibitem {Bejancu3}Bejancu, A. and Papaghiuc., Almost Semi-Invariant
Submanifolds of a Sasakian Manifold, Bull. Math, de la Soc. Sci. Math, de la
R. S. de Roumanie Tome 28 (76), nr. 1, 1984.

\bibitem {Cabras}Cabras, A. and Matzeu, P., Almost Semi-Invariant Submanifolds
of a Cosymplectic Manifold, Demonstratio Mathematica, Vol. XIX, No 2, 1916.

\bibitem {Chen}Chen, B. Y., Geometry of Submanifolds, Marcel Dekker, New York, 1973.

\bibitem {Bejancu2}Duggal, K. L. and Bejancu, A., Lightlike Submanifolds of
Codimension Two, Math. J. Toyama Univ., 15, 59--82, 1992.

\bibitem {Bejancu}Duggal, K. L. and Bejancu, A., Lightlike Submanifolds of
Semi-Riemannian Manifolds and Applications, Kluwer Academic, 364, 1996.

\bibitem {Bayrm2}Duggal, K. and Sahin, B., Screen Cauchy Riemann Lightlike
Submanifolds, Acta Math. Hungar., 106(1-2), 137--165, 2005.

\bibitem {Bayrm3}Duggal, K. L. and Sahin, B., Generalized Cauchy Riemann
Lightlike Submanifolds, Acta Math. Hungar., 112(1-2), 113--136, 2006.

\bibitem {Bayrm4}Duggal, K. L. and Sahin, B., Lightlike Submanifolds of
Indefinite Sasakian Manifolds, Int. J. Math. Math. Sci., Art ID 57585, 1--21, 2007.

\bibitem {Bayrm5}Duggal, K. L. and Sahin, B., Contact Generalized CR-Lightlike
Submanifolds of Sasakian Submanifolds, Acta Math. Hungar., 122, No. 1-2,
45--58, 2009.

\bibitem {Bayrm}Duggal, K.L. and Sahin, B., Differential Geometry of Lightlike
Submanifolds, Birkhauser Veriag AG Basel-Boston-Berlin, 2010.

\bibitem {Garima}Gupta, G., Kumar, R. and Nagaich, R. K., Geometry of
Semi-Invariant Lightlike Product Manifolds, New York J. Math. 26, 1338-1354, 2020.

\bibitem {Erol}K\i l\i\c{c}, E., \c{S}ahin, B. and Kele\c{s}, S., Screen Semi
Invariant Lightlike Submanifolds of Semi-Riemannian Product Manifolds,
International Electronic Journal of Geometry Volume 4 No. 2 pp. 120-135, 2011.

\bibitem {K1}Kazan, S., Sahin, B., Pseudosymmetric Lightlike Hypersurfaces,
Turk J. Math, 38(2014), 1050-1070.

\bibitem {K2}Kazan, S., Sahin, B., Pseudosymmetric lightlike hypersurfaces in
Indefinite Sasakian Space Forms, Journal of Applied Analysis and Computation,
6(3)(2016), 699-719.

\bibitem {K3}Kazan, S., C-Bochner Pseudosymmetric Null Hypersurfaces in
Indefinite Kenmotsu Space Forms, Acta Universitatis Apulensis, 50(2017), 111-131.

\bibitem {Papaghiuc}Papaghiuc, N., Some Results on Almaost Semi- Invariant
Submanifolds in Sasakian Manifolds, Bull. Math, de la Soc. Sci. Math, de la H.
S. de Roumanie Tome, 28 (76) nr. 3, 1984.

\bibitem {Nergiz}Poyraz, N. and Do\u{g}an B., Semi-Invariant Lightlike
Submanifolds of Golden Semi-Riemannian Manifolds, J. Baun Inst. Sci. Technol.,
22(1), 106-121, 2020.

\bibitem {Hasan}Shahid, M. H., Almost Semi-Invariant Submanifolds of
Trans-Sasakian Manifolds, Bull. Math, de la Soc.Sci.Math de Roumanie Tome, 37
(85). nr. 3-1, 1993.

\bibitem {Sahin1}\c{S}ahin, B., Screen Conformal Submersions Between Lightlike
Manifolds and Semi-Riemannian Mmanifolds and Their Harmonicity, International
Journal of Geometric Methods in Modern Physics, vol.4, pp.987-1003, 2007.

\bibitem {Sahin2}\c{S}ahin, B., Lightlike Hypersurfaces of Semi-Euclidean
Spaces Satisfying Curvature Conditions of Semisymmetry Type, Turhish Journal
of Mathematics, vol.31, pp.139-162, 2007.

\bibitem {Sahin3}\c{S}ahin, B., Slant Lightlike Submanifolds of Indefinite
Hermitian Manifolds, Balkan Journal of Geometry and its Applications, vol.13,
pp.107-119, 2008.

\bibitem {Sahin4}\c{S}ahin, B., Screen transversal lightlike submanifolds of
indefinite Kaehler manifolds, Chaos Solitons\& Fractals, vol.38, pp.1439-1448, 2008.

\bibitem {Sahin5}\c{S}ahin, B., On a Submersion Between Reinhart Lightlike
Manifolds and Semi-Riemannian Manifolds, Mediterranean Journal of Mathematics,
vol.5, pp.273-284, 2008.

\bibitem {Tripathi}Tripathi, M. M. and Sing K. D., Almost Semi-Invariant
Submanifolds of an $\varepsilon-$Framed Metric Manifold, Demonstratio
Mathematica, Vol. XXIX, No 2, 1996.
\end{thebibliography}
\end{document}